\documentclass[11pt]{amsproc}

\usepackage[hiresbb]{graphicx}
\usepackage{xcolor}

\usepackage{amsmath,amsthm,amssymb,mathtools}
\allowdisplaybreaks

\usepackage{fullpage}

\usepackage{here}
\usepackage{hyperref}
\usepackage[abbrev]{amsrefs}
\renewcommand{\MR}[1]{}

\usepackage{tikz}
\usepackage{float}
\usetikzlibrary{arrows.meta,decorations.markings}

\usepackage{aliascnt}
\usepackage[capitalize,nameinlink,noabbrev,nosort]{cleveref}

\hypersetup{
  colorlinks=true,
  linkcolor=brown,
  citecolor=brown,
  filecolor=brown,
  urlcolor=brown,
  pdftitle={Modular integrals for individual modular knots},
  pdfauthor={Toshiki Matsusaka},
  pdfsubject={2020 Mathematics Subject Classification: 11F67, 11F20, 57K10},
}

\makeatletter
\@namedef{subjclassname@2020}{\textup{2020} Mathematics Subject Classification}
\makeatother

\theoremstyle{remark}

\newaliascnt{remark}{theoremcounter}
\newtheorem{remark}[remark]{Remark}
\aliascntresetthe{remark}

\theoremstyle{definition}

\newaliascnt{definition}{theoremcounter}
\newtheorem{definition}[definition]{Definition}
\aliascntresetthe{definition}

\newaliascnt{example}{theoremcounter}
\newtheorem{example}[example]{Example}
\aliascntresetthe{example}

\theoremstyle{plain}

\newaliascnt{lemma}{theoremcounter}
\newtheorem{lemma}[lemma]{Lemma}
\aliascntresetthe{lemma}

\newaliascnt{proposition}{theoremcounter}
\newtheorem{proposition}[proposition]{Proposition}
\aliascntresetthe{proposition}

\newaliascnt{corollary}{theoremcounter}

\aliascntresetthe{corollary}

\newaliascnt{conjecture}{theoremcounter}

\aliascntresetthe{conjecture}

\newaliascnt{theorem}{theoremcounter}
\newtheorem{theorem}[theorem]{Theorem}
\aliascntresetthe{theorem}

\newaliascnt{question}{theoremcounter}

\aliascntresetthe{question}

\newaliascnt{exercise}{theoremcounter}

\aliascntresetthe{exercise}

\numberwithin{equation}{section}

\newcommand{\Z}{\mathbb{Z}}

\newcommand{\R}{\mathbb{R}}
\providecommand{\C}{}
\renewcommand{\C}{\mathbb{C}}

\newcommand{\dd}{\mathrm{d}}
\newcommand{\bbH}{\mathbb{H}}

\ExplSyntaxOn
\clist_map_inline:nn
  { A,B,C,D,E,F,G,H,I,J,K,L,M,N,O,P,Q,R,S,T,U,V,W,X,Y,Z }
  {
    \cs_new:cpn { cal#1 } { \mathcal{#1} }
  }
\ExplSyntaxOff

\ExplSyntaxOn
\clist_map_inline:nn
  { A,B,C,D,E,F,G,H,I,J,K,L,M,N,O,P,Q,R,S,T,U,V,W,X,Y,Z }
  {
    \cs_new:cpn { mat#1 } { \mathrm{#1} }
  }
\ExplSyntaxOff

\DeclareMathOperator{\ImNew}{Im}
\renewcommand{\Im}{\ImNew}
\DeclareMathOperator{\ReNew}{Re}
\renewcommand{\Re}{\ReNew}

\newcommand{\SL}{\mathrm{SL}}

\DeclareMathOperator{\sgn}{sgn}
\DeclareMathOperator{\tr}{tr}
\DeclareMathOperator{\Cyc}{Cyc}

\newcommand{\pmat}[1]{\begin{pmatrix}#1\end{pmatrix}}

\newcommand{\smat}[1]{\bigl(\begin{smallmatrix}#1\end{smallmatrix}\bigr)}

\begin{document}
\raggedbottom
\lefthyphenmin=2
\righthyphenmin=3

\title[Modular integrals for individual modular knots]{Modular integrals for individual modular knots}

\author[T.~Matsusaka]{Toshiki Matsusaka}
\address{Faculty of Mathematics, Kyushu University, Motooka 744, Nishi-ku, Fukuoka 819-0395, Japan}
\email{matsusaka@math.kyushu-u.ac.jp}


\subjclass[2020]{Primary 11F67; Secondary 11F20, 57K10}



\begin{abstract}
	To each modular knot, we attach a weight $2$ modular integral for $\SL_2(\Z)$ whose homogenized cycle integrals recover its linking numbers with other modular knots in $S^3$. This construction answers Choie--Zagier's question of explicitly constructing modular integrals with given rational period functions in weight $2$. Our formula applies to individual modular knots without symmetrization, thereby realizing an approach discussed but not pursued by Duke--Imamo\={g}lu--T\'{o}th. This gives an answer to Ghys' question on pairwise linking numbers in terms of modular integrals, as sought by Simon.
\end{abstract}

\maketitle



\section{Introduction}\label{sec-1}

In a letter to Hasse dated 25 May 1926, Hecke discussed the possible role of classical modular functions in the class field problem for real quadratic fields. In this context, he considered cycle integrals of modular functions along closed geodesics associated with real quadratic irrationalities; see \cref{fig:hecke-letter} and~\cite[Section~2.6, pp.~39--40]{HeckeHasse1926}. Specializing his formula to the classical $j$-function gives, after normalization, the cycle integrals underlying Kaneko's $\operatorname{val}$ function~\cite{Kaneko2009}. Subsequent work on this function is summarized in~\cite{Matsusaka2026Val}. In this paper, we modify these cycle integrals to construct modular integrals with given rational period functions. We then use the resulting modular integrals to compute linking numbers of individual modular knots.

\begin{figure}[ht]
\centering
	\includegraphics[width=.82\textwidth]{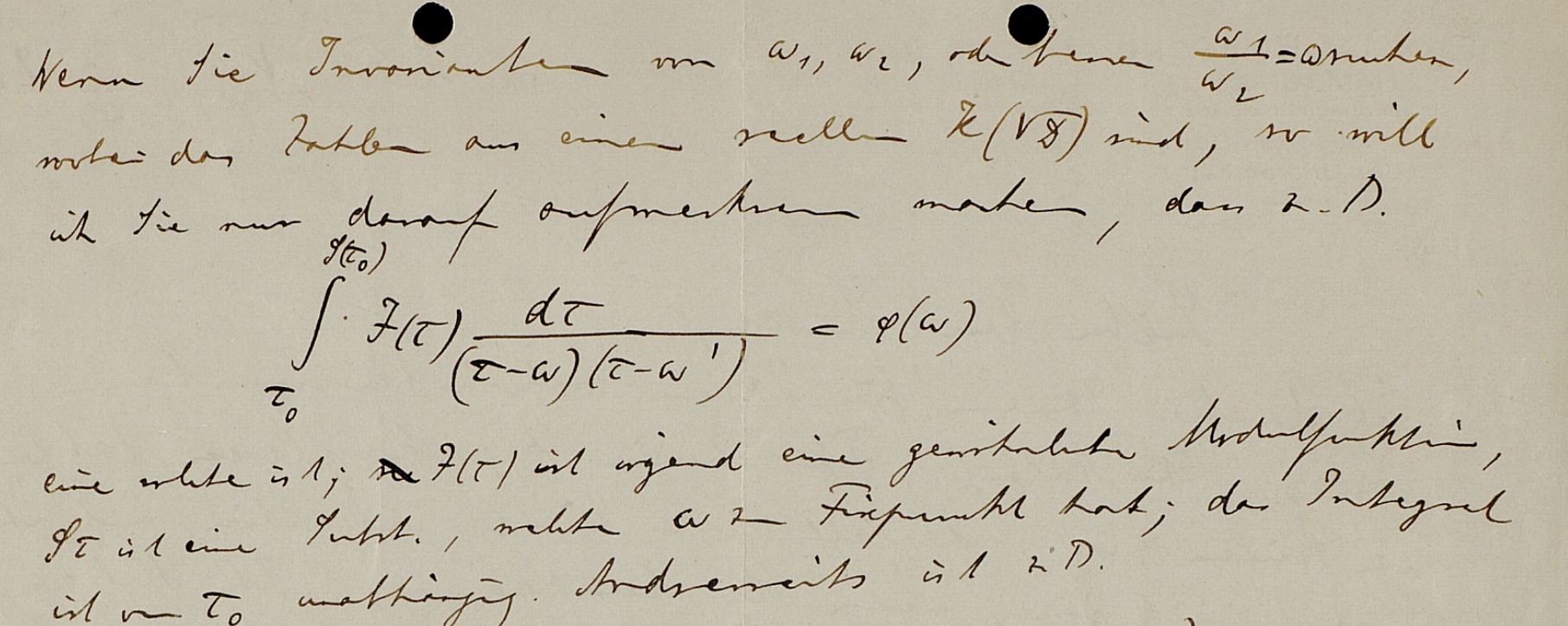}
	\caption{Hecke to Hasse, 25 May 1926, SUB G\"ottingen, Cod. Ms. H. Hasse 1:653. Image provided by SUB G\"ottingen; reproduced with permission.}
	\label{fig:hecke-letter}
\end{figure}

Let $\Gamma \coloneqq \SL_2(\Z)$ and let $\bbH$ be the upper half-plane. Fix a primitive hyperbolic $\gamma\in\Gamma$ with $\tr(\gamma)>2$. Write $w_\gamma^+$ and $w_\gamma^-$ for its attracting and repelling fixed points, and let $\overline S_\gamma$ be its closed geodesic on $\Gamma\backslash\bbH$, oriented from $w_\gamma^-$ to $w_\gamma^+$. Put $\dd s \coloneqq |\dd z|/\Im z$ and $\ell_\gamma \coloneqq \int_{\overline S_\gamma}\dd s$. For a weakly holomorphic modular function $f$, define
\begin{align}\label{eq:val-cycle}
	\calI_\gamma(f) \coloneqq \int_{\overline S_\gamma}f(z)\,\dd s.
\end{align}
Its relation to Hecke's normalization is discussed in \cref{sec2-2}. To state the generating series, let $j(z)=q^{-1}+744+O(q)$ be the classical $j$-function, where $q\coloneqq e^{2\pi iz}$. For $m\geq0$, let $j_m$ be the unique weakly holomorphic modular function with $j_m(z)=q^{-m}+O(q)$; thus $j_0=1$ and $j_1=j-744$.

Duke--Imamo\={g}lu--T\'{o}th~\cite{DIT2010,DIT2011} showed that
\begin{align}\label{eq:Hecke-generating-series}
	F_\gamma^{\mathrm{sym}}(z) \coloneqq \sum_{m=0}^\infty \calI_\gamma(j_m)q^m
\end{align}
is a modular integral of weight $2$: it is holomorphic on $\bbH$, $1$-periodic, and $(F_\gamma^{\mathrm{sym}}|_2S)-F_\gamma^{\mathrm{sym}}$ is a rational period function, where $S=\smat{0&-1\\1&0}$. Conversely, one may prescribe a rational period function and ask for a modular integral realizing it. Knopp~\cite{Knopp1974,Knopp1978} proved that such a modular integral always exists. Choie--Zagier~\cite{ChoieZagier1993} later classified rational period functions in terms of indefinite binary quadratic forms and asked for explicit constructions of the corresponding modular integrals. We review the relevant background on rational period functions in \cref{sec2}.

A family in the weight $2$ classification of Choie--Zagier that will be central here is
\begin{align}\label{eq:rho-gamma}
	\rho_\gamma(z) \coloneqq 2\sum_{\substack{h\sim\gamma\\w_h^-w_h^+<0}} \frac{\sgn(w_h^+)}{z-w_h^+} = 2\sum_{\substack{h\sim\gamma\\w_h^-<0<w_h^+}} \left(\frac{1}{z-w_h^+}-\frac{1}{z+1/w_h^+} \right),
\end{align}
where $h\sim\gamma$ denotes conjugacy in $\Gamma$. The Duke--Imamo\={g}lu--T\'{o}th modular integral in~\eqref{eq:Hecke-generating-series} satisfies
\[
	(F_\gamma^{\mathrm{sym}}|_2S)-F_\gamma^{\mathrm{sym}} = \frac12\bigl(\rho_\gamma+\rho_{\gamma^{-1}}\bigr).
\]
The symmetrization occurs because $\calI_{\gamma^{-1}}(f)=\calI_\gamma(f)$: hyperbolic arclength is insensitive to orientation. To distinguish the two orientations, we therefore add to the classical cycle integral a correction term that changes sign under $\gamma\mapsto\gamma^{-1}$.

The correction term is expressed in terms of the weight $2$ Eisenstein series and its non-holomorphic completion\footnote{Following~\cite{DIT2017}, we use $\sigma$ for elements of $\Gamma$ later; no confusion with the divisor function $\sigma_1(m)$ should arise.}:
\[
	E_2(z) \coloneqq 1-24\sum_{m=1}^\infty \sigma_1(m)q^m,
	\qquad E_2^*(z) \coloneqq E_2(z)-\frac{3}{\pi\Im z},
	\qquad \sigma_1(m) \coloneqq \sum_{d\mid m}d.
\]
For the normalization relevant to the linking-number formula, we seek a function satisfying
\begin{align}\label{eq:linking-trans}
	F(z+1)=F(z), \qquad (F|_2S)(z)-F(z) = \rho_\gamma(z)+\frac{2\Psi(\gamma)}{z},
\end{align}
where $\Psi$ is the Rademacher symbol, with the convention recalled in \cref{sec2-2}.

\begin{theorem}\label{thm:A}
	There is a unique function $F_\gamma^{\mathrm{lk}}$, holomorphic on $\bbH$ and at $i\infty$, satisfying~\eqref{eq:linking-trans}. Its Fourier expansion is
	\begin{align}\label{eq:F-lk}
		F_\gamma^{\mathrm{lk}}(z)=\ell_\gamma+\sum_{m=1}^{\infty}b_\gamma(m)q^m,
	\end{align}
	where
	\begin{align}\label{eq:b-gamma}
		b_\gamma(m) = \int_{\overline S_\gamma}j_m(z)\,\dd s -\frac{\pi i}{3} \int_{\overline S_\gamma} \bigl(j_m(z)+24\sigma_1(m)\bigr)E_2^*(z)\,\dd z.
	\end{align}
\end{theorem}

The first term in~\eqref{eq:b-gamma} is unchanged under $\gamma\mapsto\gamma^{-1}$, while the second changes sign. Thus averaging the two modular integrals recovers the Duke--Imamo\={g}lu--T\'{o}th series:
\[
	\frac12\bigl(F_\gamma^{\mathrm{lk}} + F_{\gamma^{-1}}^{\mathrm{lk}}\bigr) = F_\gamma^{\mathrm{sym}}.
\]
The additional term $2\Psi(\gamma)/z$ in~\eqref{eq:linking-trans} can be removed by setting
\[
	F_\gamma(z) = F_\gamma^{\mathrm{lk}}(z) -\frac{\pi i}{3}\Psi(\gamma)E_2(z).
\]
Since $(E_2|_2S)(z)-E_2(z)=6/(\pi iz)$, this function satisfies
\[
	(F_\gamma|_2S)(z)-F_\gamma(z) = \rho_\gamma(z).
\]
This provides an explicit solution to Choie--Zagier's construction problem in weight $2$.

We now turn to modular knots. The $3$-manifold $\Gamma\backslash\SL_2(\R)$ is homeomorphic to the complement of a trefoil $\calT$ in $S^3$. The geodesic flow on this manifold is given by right multiplication by $\operatorname{diag}(e^{t/2},e^{-t/2})$. Following Ghys~\cite{Ghys2007}, we call its periodic orbits, oriented by increasing $t$, \emph{modular knots}. They are indexed by primitive hyperbolic conjugacy classes with positive trace, and we write $C_\sigma$ for the knot corresponding to $\sigma$. Although $\sigma$ and $\sigma^{-1}$ project to the same closed geodesic with opposite orientations, they generally give distinct periodic orbits in $\Gamma \backslash \SL_2(\R)$. 

Writing $\mathrm{Lk}$ for the linking number in $S^3$, Ghys proved that, for a suitable orientation of the trefoil,
\begin{align}\label{eq:Ghys-linking}
	\mathrm{Lk}(C_\sigma,\calT)=\Psi(\sigma).
\end{align}
He also asked for an arithmetical or combinatorial computation of linking numbers between modular knots~\cite[p.~273]{Ghys2007}. Duke--Imamo\={g}lu--T\'{o}th~\cite{DIT2017} expressed the linking numbers of disjoint symmetrized cycles $C_\gamma+C_{\gamma^{-1}}$ in terms of $F_\gamma^{\mathrm{sym}}$. These cycles are null-homologous in the trefoil complement, and the corresponding formula is recalled in~\eqref{eq:sym-linking}. The same work also discusses an approach to individual knots in $S^3$, leaving open its connection with modular forms. Simon~\cite{Simon2025} subsequently computed linking numbers of individual modular knots and asked for a modular interpretation of his formula~\cite[Section~8.2]{Simon2025}.

To explain the form of our answer, we first express Ghys's formula as an integral of $E_2$. The identity
\[
	2\pi iE_2(z)\,\dd z=\dd\log\Delta(z), 
	\qquad \Delta(z)=q\prod_{n=1}^\infty (1-q^n)^{24},
\]
shows that the imaginary part of the integral measures the change of argument of $\Delta$. More precisely, for any $z_0\in\bbH$ and hyperbolic $\sigma$ with $\tr(\sigma)>2$, one has~\cite[(1.3)]{Matsusaka2024} 
\begin{align}\label{eq:classical-hcyc}
	\Psi(\sigma)=\frac{1}{2\pi}\Im\lim_{N\to\infty} \int_{\sigma^Nz_0}^{\sigma^{N+1}z_0} 2\pi iE_2(z)\,\dd z.
\end{align}
Because $E_2(z)\,\dd z$ is not $\Gamma$-invariant, the integral from $z$ to $\sigma z$ generally depends on $z$. Taking $z = \sigma^Nz_0\to w_\sigma^+$ yields a limit independent
of $z_0$, called a \emph{homogenized cycle integral}. Our second theorem replaces $2\pi iE_2$ by $F_\gamma^{\mathrm{lk}}$ to compute linking with $C_\gamma$ rather than with the trefoil.

\begin{theorem}\label{thm:B}
	Let $\sigma\in\Gamma$ be primitive hyperbolic with $\tr(\sigma)>2$ and not conjugate to $\gamma$. Then, for any $z_0\in\bbH$,
	\[
		\mathrm{Lk}(C_\gamma,C_\sigma) = \frac{1}{2\pi}\Im\lim_{N\to\infty} \int_{\sigma^Nz_0}^{\sigma^{N+1}z_0} F_\gamma^{\mathrm{lk}}(z)\,\dd z.
	\]
\end{theorem}

The remainder of the paper is organized as follows. In \cref{sec2}, we prove \cref{thm:A}. A key ingredient is the differential on $\overline S_\gamma$ induced by the holomorphic differential $\omega_\gamma(z)\,\dd z$ constructed in \cref{lem:omega-h}. It leads to the modular integral $F_\gamma$ and to explicit cycle-integral formulas for its Fourier coefficients. In \cref{sec3}, we derive a combinatorial formula for the homogenized cycle integrals and compare it with Simon's linking-number formula to prove \cref{thm:B}.

\section{Rational period functions and modular integrals of weight \texorpdfstring{$2$}{2}}\label{sec2}

For $k\geq1$ and $\sigma=\smat{a&b\\c&d}\in\Gamma$, the weight $2k$ slash action is
\[
(F|_{2k}\sigma)(z)\coloneqq(cz+d)^{-2k}F\!\left(\frac{az+b}{cz+d}\right),
\]
extended linearly to sums of matrices. Recall $S \coloneqq \smat{0&-1\\1&0}$, $T \coloneqq \smat{1&1\\0&1}$, and $U \coloneqq TS=\smat{1&-1\\1&0}$.

\begin{definition}\label{def:rpf}
	A \emph{rational period function of weight $2k$} on $\Gamma$ is a rational function $\rho(z)$ satisfying
	\begin{align}\label{eq:rpf-relations}
		\rho|_{2k}(1+S)=0,\qquad \rho|_{2k}(1+U+U^2)=0.
	\end{align}
	A \emph{modular integral of weight $2k$} with rational period function $\rho$ is a holomorphic function $F$ on $\bbH$ satisfying
	\begin{align}\label{eq:modular-int-trans}
		F(z+1)=F(z),\qquad (F|_{2k}S)(z)-F(z)=\rho(z).
	\end{align}
\end{definition}

For any holomorphic $1$-periodic function $F$ on $\bbH$, the function $\rho \coloneqq F|_{2k}S-F$ satisfies~\eqref{eq:rpf-relations}, since $F|_{2k}U=F|_{2k}S$, $S^2=U^3=-I$, and $-I$ acts trivially in even weight. If $\rho$ is rational, then $F$ is a modular integral of weight $2k$ with rational period function $\rho$. Conversely, Knopp~\cite{Knopp1974} proved that every rational period function $\rho$ arises in this way.

In weight $2$, Choie--Zagier's classification~\cite{ChoieZagier1993} gives a $\C$-basis consisting of $12/z$, $z^{-2}-1$, and the functions $\rho_\gamma$ in~\eqref{eq:rho-gamma}. The first two basis elements are realized by $2\pi iE_2$ and $1$, respectively. It therefore remains to construct modular integrals with period functions $\rho_\gamma$. To this end, we first extend $\rho_\gamma$ to a strongly parabolic cocycle on $\Gamma$. For certain rational period functions of weights greater than $2$, Parson~\cite{Parson1993} constructed modular integrals using partial hyperbolic Eisenstein series. Duke--Imamo\={g}lu--T\'{o}th~\cite{DIT2010,DIT2011} gave a cycle-integral construction in higher weights, which agrees with Parson's construction up to cusp forms. In weight $2$, their construction gives the modular integral $F_\gamma^{\mathrm{sym}}$, which the author~\cite[Theorem~1.1]{Matsusaka2024} recovered by regularizing Parson's series.

For the construction of modular integrals, it is convenient to keep track of their transformation terms under every element of $\Gamma$, rather than only under $T$ and $S$. A \emph{rational cocycle of weight $2$} is a map $\sigma \mapsto r(\sigma, \cdot)$ from $\Gamma$ to the space of rational functions whose finite poles lie on $\R$, satisfying
\[
	r(\sigma_1 \sigma_2, z) = (r(\sigma_1, \cdot)|_2 \sigma_2)(z) + r(\sigma_2, z), \qquad \sigma_1, \sigma_2 \in \Gamma.
\]
We impose the \emph{strongly parabolic} normalization $r(T, z) = 0$ following~\cite[Section~2]{DIT2017}. To relate these cocycles to rational period functions, put $\rho(z) \coloneqq r(S, z)$. The cocycle identity gives $r(I,z) = 0$. Since $-I$ acts trivially in weight $2$, it also gives $2r(-I, z) = r(I, z) = 0$. Moreover, $U = TS$ and $r(T, z) = 0$ imply $r(U, z) = \rho(z)$. The relations $S^2 = U^3 = -I$ therefore give
\[
	\rho|_2(1+S)(z) = r(S^2, z) = 0, \qquad \rho|_2(1+U+U^2)(z) = r(U^3, z) = 0.
\]
Thus $\rho$ is a rational period function of weight $2$. Conversely, a rational period function $\rho$ determines $r$ uniquely, since $S$ and $T$ generate $\Gamma$ and $r(T, z) = 0$. Under this correspondence, a modular integral $F$ with $\rho$ satisfies
\[
	(F|_2 \sigma)(z) - F(z) = r(\sigma, z), \qquad \sigma \in \Gamma.
\]

\subsection{Modular cocycles $r_\gamma(\sigma,z)$}\label{sec2-1}

For a hyperbolic $h = \smat{a & b \\ c & d} \in \Gamma$ with $\tr(h) > 2$, write
\[
	Q_h(X, Y) \coloneqq cX^2 + (d-a) XY - bY^2, \qquad D_h \coloneqq \tr(h)^2 - 4, \qquad \epsilon_h \coloneqq \sgn(c).
\]
Here $D_h$ is the discriminant of $Q_h$. Hyperbolicity forces $c \neq 0$, since $c = 0$ would give $ad = 1$ and $|\tr(h)| = 2$. Write $w_h > w'_h$ for the fixed points in decreasing order. The attracting point $w_h^+$ is $w_h$ for $\epsilon_h = 1$ and $w'_h$ for $\epsilon_h = -1$. Let $S_h$ be the axis of $h$, oriented from $w_h^-$ to $w_h^+$; see \cref{fig:axis-orientation}.

\begin{figure}[H]
  \centering
  \begin{tikzpicture}[
    line cap=round,
    line join=round,
    axis/.style={
      -{Stealth[length=5pt, width=4pt]},
      line width=0.5pt
    },
    oriented arc/.style={
      line width=0.9pt,
      decoration={
        markings,
        mark=at position 0.5 with {
          \arrow{Stealth[length=6pt, width=5pt]}
        }
      },
      postaction={decorate}
    },
    point/.style={
      circle,
      fill=black,
      inner sep=0pt,
      minimum size=3.5pt
    },
    heading/.style={
      draw,
      line width=0.5pt,
      inner xsep=8pt,
      inner ysep=4pt,
      anchor=south
    },
    arc label/.style={
      above right=2pt,
      inner sep=1pt
    }
  ]

    \begin{scope}

      \draw[axis] (-2,0) -- (2,0)
        node[right] {$\mathbb{R}$};

      \draw[oriented arc]
        (-1.4,0)
        arc[start angle=180, end angle=0, radius=1.4];

      \node[arc label] at (45:1.4) {$S_h$};

      \node[heading, yshift=8pt] at (0,1.4)
        {$\epsilon_h=1 \quad (c>0)$};

      \node[point] at (-1.4,0) {};
      \node[point] at (1.4,0) {};

      \node[below=5pt] at (-1.4,0) {$w'_h=w_h^-$};
      \node[below=5pt] at (1.4,0)  {$w_h=w_h^+$};

    \end{scope}

    \begin{scope}[xshift=5cm]

      \draw[axis] (-2,0) -- (2,0)
        node[right] {$\mathbb{R}$};

      \draw[oriented arc]
        (1.4,0)
        arc[start angle=0, end angle=180, radius=1.4];

      \node[arc label] at (45:1.4) {$S_h$};

      \node[heading, yshift=8pt] at (0,1.4)
        {$\epsilon_h=-1 \quad (c<0)$};

      \node[point] at (-1.4,0) {};
      \node[point] at (1.4,0) {};

      \node[below=5pt] at (-1.4,0) {$w'_h=w_h^+$};
      \node[below=5pt] at (1.4,0)  {$w_h=w_h^-$};

    \end{scope}

  \end{tikzpicture}
\caption{The sign $\epsilon_h$ records the direction of the oriented axis.}\label{fig:axis-orientation}
\end{figure}

For primitive $h$, the stabilizer $\Gamma_h \coloneqq \{g \in \Gamma : g w_h^\pm = w_h^\pm\}$ is given by $\Gamma_h = \{\pm h^n : n \in \Z\}$. We index conjugates $h = g^{-1} \gamma g$ by $g \in \Gamma_\gamma \backslash \Gamma$. We use the right action of $\Gamma$ on binary quadratic forms given by $(Q \circ \smat{a & b\\ c & d})(X, Y) \coloneqq Q(aX+bY, cX+dY)$. This action is compatible with conjugation, that is, $Q_{g^{-1} \gamma g} = Q_\gamma \circ g$ (see also~\cite[Section~2]{Matsusaka2024}). 

For $\sigma = \smat{a & b \\ c & d} \in \Gamma$ with $c \neq 0$, let
\[
	I_\gamma(\sigma) \coloneqq \{g \in \Gamma_\gamma \backslash \Gamma : w'_{g^{-1}\gamma g} < -d/c < w_{g^{-1} \gamma g}\}.
\]
This indexes the axes crossing the geodesic from $\sigma^{-1} i\infty$ to $i\infty$; see \cref{fig:crossing-set}. For $c = 0$, we set $I_\gamma(\sigma) \coloneqq \varnothing$.

\begin{figure}[H]
  \centering
\begin{tikzpicture}[
  x=1.75cm, y=1.75cm,
  line cap=round, line join=round,
  net/.style={draw=black!45, line width=0.28pt},
  crossing/.style={draw=black, line width=0.8pt},
  axis/.style={-{Stealth[length=4pt,width=3pt]}, line width=0.5pt},
  cut/.style={densely dashed, line width=0.5pt,
              -{Stealth[length=4pt,width=3pt]}}
]
  \def\xmin{-4.05}
  \def\xmax{3.05}
  \pgfmathsetmacro{\xcut}{-1/3}
  \begin{scope}
    \clip (\xmin,0) rectangle (\xmax,1.96);

    \foreach \A/\B in {
      1/0,
      2/-2,
      3/0,
      6/-6,
      11/-10,
      11/10,
      13/-8,
      13/8,
      22/-10,
      22/10,
      23/-14,
      23/14,
      26/-18,
      26/18,
      33/-12,
      33/12,
      37/-30,
      37/30,
      39/-18,
      39/18,
      46/-14,
      46/14,
      47/-24,
      47/24,
      59/-22,
      59/22
    } {
      \pgfmathsetmacro{\xc}{-\B/(2*\A)}
      \pgfmathsetmacro{\rad}{sqrt(3)/\A}
      \foreach \k in {-10,...,10} {
        \draw[net] ({\xc+\k+\rad},0)
          arc[start angle=0,end angle=180,radius=\rad];
      }
    }

    \foreach \A/\B/\C in {1/-2/-2, 1/0/-3, 1/2/-2, 1/4/1, 2/-2/-1, 2/2/-1, 3/0/-1, 6/6/1, 11/10/2, 13/8/1, 23/14/2, 26/18/3} {
      \pgfmathsetmacro{\xc}{-\B/(2*\A)}
      \pgfmathsetmacro{\rad}{sqrt(3)/\A}
      \draw[crossing] ({\xc+\rad},0)
        arc[start angle=0,end angle=180,radius=\rad];
    }
  \end{scope}

  \draw[axis] ({\xmin-0.045},0)
    -- ({\xmax+0.075},0)
    node[right=2pt] {$\mathbb{R}$};
  \draw[cut] (\xcut,0) -- (\xcut,1.925) node[above=3pt] {$i\infty$};
  \node[anchor=north,yshift=-5pt,inner sep=1pt]
    at (\xcut,0) {$\sigma^{-1} i\infty = -\dfrac{d}{c}$};
\end{tikzpicture}
\caption{The axes counted by $I_\gamma(\sigma)$ cross the dashed geodesic.}\label{fig:crossing-set}
\end{figure}

With the positive-trace convention, $\epsilon_h=\sgn(w_h^+-w_h^-)$. Thus the product $\epsilon_{\sigma h\sigma^{-1}}\epsilon_h$ records whether $\sigma$ preserves or reverses the order of the two endpoints on $\R$. For $c\ne0$, this order is reversed exactly when the pole $-d/c$ lies between them, that is, when the axis crosses the reference geodesic in \cref{fig:crossing-set}.

\begin{lemma}\label{lem:epsilon-prop}
	The crossing set $I_\gamma(\sigma)$ is finite, and for $h = g^{-1} \gamma g$, we have
	\[
		\epsilon_{\sigma h \sigma^{-1}} \epsilon_h = \begin{cases}
			-1 &\text{if } g \in I_\gamma(\sigma),\\
			+1 &\text{if } g \not\in I_\gamma(\sigma).
		\end{cases}
	\]
	Moreover,
	\begin{align}\label{eq:sign-cancel}
		\sum_{g \in I_\gamma(\sigma)} \epsilon_{g^{-1}\gamma g} = 0.
	\end{align}
\end{lemma}

\begin{proof}
	The finiteness and the sign criterion follow from \cite[Proposition 2.6 and its proof]{Matsusaka2024}. For the last assertion, put
	\[
		\nu(\sigma) \coloneqq \frac{1}{2} \sum_{h \sim \gamma} (\epsilon_h - \epsilon_{\sigma h \sigma^{-1}}) = \sum_{g \in I_\gamma(\sigma)} \epsilon_{g^{-1} \gamma g}.
	\]
	Then $\nu$ defines a homomorphism $\Gamma\to\Z$ by reindexing. Since $\Gamma$ is generated by the finite-order elements $S$ and $U$, we have $\nu=0$, proving \eqref{eq:sign-cancel}.
\end{proof}

\begin{remark}
	For $c\ne0$, this vanishing also has a geometric interpretation. Orient $L_\sigma$ from $\sigma^{-1}i\infty=-d/c$ to $i\infty$. Since both endpoints represent the same cusp, its projection, completed at the cusp, defines a closed cycle $\overline L_\sigma$ on $X(1)$. The sum in~\eqref{eq:sign-cancel} is the algebraic intersection number of $\overline S_\gamma$ and $\overline L_\sigma$, which vanishes because $H_1(X(1),\Z)=0$.
\end{remark}

To retain the direction of each axis, we place a pole at its attracting fixed point. The sign records the direction in which that axis crosses the reference geodesic.

\begin{definition}\label{def:r-gamma}
	The (unsymmetrized) modular cocycle attached to $\gamma$ is
	\[
		r_\gamma(\sigma, z) \coloneqq 2 \sum_{g \in I_\gamma(\sigma)} \frac{\epsilon_{g^{-1}\gamma g}}{z- w_{g^{-1}\gamma g}^+}.
	\]
\end{definition}

\begin{proposition}\label{prop:modular-cocycle}
	The function $r_\gamma$ is a strongly parabolic cocycle of weight $2$.
\end{proposition}

\begin{proof}
	By \cref{lem:epsilon-prop},
	\begin{align}\label{eq:r-reexp}
		r_\gamma(\sigma, z) = \sum_{h \sim \gamma} \frac{\epsilon_h - \epsilon_{\sigma h \sigma^{-1}}}{z - w_h^+},
	\end{align}
	with only finitely many nonzero terms. Write $\sigma_2 = \smat{a&b\\c&d} \in \Gamma$. Since conjugation carries attracting fixed points to attracting fixed points, we have $w_{\sigma_2^{-1}h\sigma_2}^+ = \sigma_2^{-1}w_h^+$ and hence
	\[
		\left(\frac{1}{z-w_h^+}\right)\Big|_2\sigma_2 = \frac{1}{z-w_{\sigma_2^{-1}h\sigma_2}^+}-\frac{c}{cz+d}.
	\]
	Using this identity, reindexing $h\mapsto\sigma_2h\sigma_2^{-1}$, and applying~\eqref{eq:sign-cancel}, we obtain
	\begin{align*}
		(r_\gamma(\sigma_1,\cdot)|_2\sigma_2)(z) &= \sum_{h\sim\gamma} \frac{\epsilon_{\sigma_2h\sigma_2^{-1}} - \epsilon_{\sigma_1\sigma_2h(\sigma_1\sigma_2)^{-1}}}{z-w_h^+}.
	\end{align*}
	Adding
	\[
		r_\gamma(\sigma_2,z) = \sum_{h\sim\gamma} \frac{\epsilon_h-\epsilon_{\sigma_2h\sigma_2^{-1}}}{z-w_h^+}
	\]
	gives $r_\gamma(\sigma_1\sigma_2,z)$. Finally, $I_\gamma(T)=\varnothing$ gives $r_\gamma(T,z)=0$.
\end{proof}

In particular, $\rho_\gamma(z) = r_\gamma(S,z)$ defined in~\eqref{eq:rho-gamma} is a rational period function of weight $2$. Replacing $h$ by $h^{-1}$ exchanges its attracting and repelling fixed points and changes both $Q_h$ and $\epsilon_h$ to their negatives. For $h = \smat{a&b\\c&d}\sim\gamma$, we have
\[
	Q_h(z,1)=c(z-w_h^+)(z-w_h^-), \qquad c(w_h^+-w_h^-)=\sqrt{D_\gamma}.
\]
Consequently,
\[
    \epsilon_h\left(\frac{1}{z-w_h^+}-\frac{1}{z-w_h^-}\right) =\frac{\epsilon_h\sqrt{D_\gamma}}{Q_h(z,1)}.
\]
Pairing the terms indexed by $h$ and $h^{-1}$ therefore gives
\[
	r_\gamma^{\mathrm{sym}}(\sigma, z) \coloneqq \frac{1}{2} \bigg(r_\gamma(\sigma, z) + r_{\gamma^{-1}}(\sigma, z)\bigg) = \sqrt{D_\gamma} \sum_{g \in I_\gamma(\sigma)} \frac{\epsilon_{g^{-1} \gamma g}}{Q_{g^{-1} \gamma g}(z,1)}.
\]
This is the modular cocycle of~\cite{DIT2017}.

\begin{remark}
	The factor $\epsilon_h$ is responsible for the transformation defect in both the cocycle above and the hyperbolic Poincar\'e series of~\cite[(2.2) and Proposition~2.6]{Matsusaka2024}. In the latter, the Parson-type series contains the factor $\epsilon_h$. In higher weights, the corresponding series without this sign factor are studied by Katok~\cite{Katok1985}; a related weight $2$ construction with values in differential forms was given by Kudla--Millson~\cite{KudlaMillson1979}. With the factor present, reindexing changes the signs precisely for the crossing axes.
\end{remark}

\subsection{Cycle integrals}\label{sec2-2}

As a first step toward constructing a modular integral with period function $\rho_\gamma$, we express the cycle integral~\eqref{eq:val-cycle} in terms of a holomorphic differential. Each integral over $\overline S_\gamma$ is taken over one period of the parametrized geodesic. On $S_\gamma$ with its chosen orientation, we have
\[
	\dd s = -(w_\gamma^+-w_\gamma^-)\, \frac{\dd z}{(z-w_\gamma^+)(z-w_\gamma^-)} = \left(\frac{1}{z-w_\gamma^-} - \frac{1}{z- w_\gamma^+}\right)\, \dd z.
\]
Consequently, for any $z_0\in\bbH$ and weakly holomorphic modular function $f$,
\begin{align}\label{eq:Hecke-normalization}
	\calI_\gamma(f) = -(w_\gamma^+-w_\gamma^-) \int_{z_0}^{\gamma z_0} f(z)\,\frac{\dd z}{(z-w_\gamma^+)(z-w_\gamma^-)}.
\end{align}
The differential on the right is holomorphic on $\bbH$ and invariant under $\gamma$. Its integral is therefore independent of both the path and the base point.

We use the following convention for the Rademacher symbol. The nonvanishing function $\Delta$ has a holomorphic logarithm on $\bbH$. For $\sigma\in\Gamma$, put
\begin{align}\label{eq:def-Phi-Psi}
	\Phi(\sigma) \coloneqq \frac{1}{2\pi} \lim_{y\to\infty} \Im\Bigl(\log\Delta(\sigma iy)-\log\Delta(iy)\Bigr), 
	\qquad \Psi(\sigma)\coloneqq\lim_{n\to\infty}\frac{\Phi(\sigma^n)}{n}.
\end{align}
The passage from $\Phi$ to $\Psi$ is called \emph{homogenization}. The resulting symbol is conjugacy invariant and satisfies $\Psi(\sigma^{-1}) = -\Psi(\sigma)$, so that it is sensitive to orientation. It also vanishes on elliptic elements~\cite[(1.4), (1.6), and (1.8)]{DIT2017}. See also~\cite{Rademacher1972,MatsusakaUeki2023}. This normalization agrees with~\eqref{eq:classical-hcyc}.

For $\sigma = \smat{a & b \\ c & d} \in \Gamma$, the Eisenstein series satisfy
\begin{align}\label{eq:Eisen-trans}
	(E_2|_2 \sigma)(z) - E_2(z) = \frac{6}{\pi i} \frac{c}{cz+d}, \qquad E_2^*|_2 \sigma = E_2^*.
\end{align}
For primitive hyperbolic $h$ with $\tr(h) > 2$, the invariant differential $E_2^*(z)\, \dd z$ therefore descends to the oriented circle $\overline{S}_h = \Gamma_h \backslash S_h$. Its integral, taken along one fundamental segment of $S_h$, gives the Rademacher symbol (see~\cite[Theorem~B~(2)]{MatsusakaUeki2023}):
\begin{align}\label{eq:Rade-cycle}
	\Psi(h) = \int_{\overline{S}_h} E_2^*(z) \,\dd z.
\end{align}
Thus, unlike the hyperbolic arclength integral~\eqref{eq:val-cycle}, the Rademacher cycle integral changes sign when the orientation of the geodesic is reversed.

This contrast suggests a modification of Hecke's cycle integral that retains the orientation of the axis. We begin with the simple-pole differential $\dd z/(z-w_h^+)$, which distinguishes the attracting endpoint from the repelling one. By itself, however, this differential is not invariant under the stabilizer of the axis. The extra term in the transformation law of $E_2$ cancels precisely this failure of invariance. This leads to a holomorphic differential $\omega_h(z)\,\dd z$ whose restriction to $S_h$ descends to $\overline S_h$.

\begin{lemma}\label{lem:omega-h}
	For the above $h$, the holomorphic function
	\[
		\omega_h(z) \coloneqq \frac{2}{z-w_h^+} + \frac{\pi i}{3} E_2(z)
	\]
	satisfies $\omega_h|_2 \sigma = \omega_{\sigma^{-1} h \sigma}$ for every $\sigma \in \Gamma$.
\end{lemma}

\begin{proof}
	The two terms acquire respectively $-2c/(cz+d)$ and $2c/(cz+d)$ under the slash operator, and these cancel.
\end{proof}

Replacing the hyperbolic arclength element $\dd s$ in~\eqref{eq:val-cycle} by $-\omega_\gamma(z)\,\dd z$ leads to the following modified cycle integrals.

\begin{definition}
	For $m \ge 0$, set
	\[
		a_\gamma(m) \coloneqq -\int_{\overline{S}_\gamma} j_m(z) \omega_\gamma(z) \, \dd z.
	\]
\end{definition}

By \cref{lem:omega-h}, $a_\gamma(m)$ depends only on the conjugacy class of $\gamma$.

\begin{proposition}\label{prop:a-gamma}
	For every $m \ge 0$,
	\[
		a_\gamma(m) = \int_{\overline{S}_\gamma} j_m(z)\, \dd s - \frac{\pi i}{3} \int_{\overline{S}_\gamma} j_m(z) E_2^*(z) \, \dd z.
	\]
	In particular, $a_\gamma(0) = \ell_\gamma - \pi i \Psi(\gamma)/3$.
\end{proposition}

\begin{proof}
	On $S_\gamma$, with $Q_\gamma = [A, B, C] \coloneqq AX^2 + BXY + CY^2$ and $z = x+iy$, we have $A|z|^2 + Bx + C = 0$, and therefore
	\[
		\frac{\partial_z Q_\gamma(z,1)}{Q_\gamma(z,1)} = -\frac{i}{y}, \qquad \frac{\sqrt{D_\gamma}}{Q_\gamma(z,1)} + \frac{\partial_z Q_\gamma(z,1)}{Q_\gamma(z,1)} = \frac{2}{z - w_\gamma^+}.
	\]
	On the axis, $-\omega_\gamma(z)$ is therefore the sum of $-\sqrt{D_\gamma}/Q_\gamma(z,1)$ and $-(\pi i/3) E_2^*(z)$. With its chosen orientation, $-\sqrt{D_\gamma} \, \dd z/Q_\gamma(z,1) = |\dd z|/y$, proving the formula. The constant term follows from $j_0 = 1$.
\end{proof}

\subsection{Modular integrals $F_\gamma(z)$}\label{sec2-3}

We now show that the series with Fourier coefficients $a_\gamma(m)$ gives the desired modular integral with cocycle $r_\gamma$. For $h \sim \gamma$, write $Q_h = [A, B, C]$ and 
\[
	u_h(z) \coloneqq A|z|^2 + B \Re z + C.
\]
For $\sigma = \smat{a & b \\ c & d} \in \Gamma$, the relation $Q_{\sigma^{-1} h\sigma} = Q_h \circ \sigma$ gives $u_{\sigma^{-1} h\sigma}(z) = |cz+d|^2 u_h(\sigma z)$. In particular, off the corresponding axes, 
\begin{align}\label{eq:u-trans}
	\sgn u_{\sigma^{-1} h \sigma}(z) = \sgn u_h(\sigma z).
\end{align}
Let $\chi_h(z)$ be the indicator of the open half-disk $B_h$ bounded by $S_h$ and the real line. Off $S_h$, we have
\begin{align}\label{eq:ep-chi}
	\epsilon_h \chi_h(z) = \frac{1}{2} \bigg(\epsilon_h - \sgn u_h(z) \bigg).
\end{align}

\begin{lemma}\label{lem:local-finite}
	For every compact set $K \subset \bbH$, only finitely many $h \sim \gamma$ satisfy $\overline{B}_h \cap K \neq \varnothing$. Each half-disk has height at most $\sqrt{D_\gamma}/2$.
\end{lemma}

\begin{proof}
	If $z = x+iy \in K \cap \overline{B}_h$, then
	\[
		(2Ax + B)^2 + 4A^2 y^2 \le D_\gamma.
	\]
	On a compact subset $K$, this bounds the nonzero integer $A$ and then $B$. The coefficient $C = (B^2 - D_\gamma)/(4A)$ and the matrix $h$ are then determined. The height bound follows from $|A| \ge 1$.
\end{proof}

\begin{theorem}\label{thm:modular-int-Fgamma}
	The series $F_\gamma(z) \coloneqq \sum_{m \ge 0} a_\gamma(m) q^m$ is the unique function holomorphic on $\bbH$ and at $i\infty$ satisfying
	\[
		(F_\gamma|_2 \sigma)(z) - F_\gamma(z) = r_\gamma(\sigma, z)
	\]
	for every $\sigma \in \Gamma$.
\end{theorem}

\begin{proof}
	By \cref{prop:modular-cocycle}, $r_\gamma$ is a strongly parabolic cocycle. Here Knopp's coefficient space $\calP$ consists of holomorphic functions $\varphi$ on $\bbH$ satisfying $|\varphi(z)|\le C(|z|^\alpha+(\Im z)^{-\beta})$ for some $C,\alpha,\beta>0$. Since $r_\gamma(\sigma,z)\ll_{\gamma,\sigma}(\Im z)^{-1}$, each $r_\gamma(\sigma,\cdot)$ belongs to $\calP$. Applying \cite[Theorem~3]{Knopp1974} and removing the cusp principal part as described in \cite[p.~622]{Knopp1974}, we obtain a function $\widetilde{F}_\gamma(z)$, holomorphic on $\bbH$ and at $i\infty$, such that
	\[
		\widetilde{F}_\gamma(z) = \sum_{m=0}^\infty \widetilde{a}_\gamma(m) q^m, \qquad (\widetilde{F}_\gamma|_2 \sigma)(z) - \widetilde{F}_\gamma(z) = r_\gamma(\sigma, z)
	\]
	for every $\sigma \in \Gamma$.
	
	By \cref{lem:local-finite}, the sum $\calB_\gamma(z) \coloneqq \sum_{h \sim \gamma} \epsilon_h \chi_h(z) \omega_h(z)$ is locally finite and vanishes above height $\sqrt{D_\gamma}/2$. Using~\eqref{eq:sign-cancel}, \eqref{eq:r-reexp}, \eqref{eq:u-trans}, \eqref{eq:ep-chi}, and \cref{lem:omega-h}, we obtain, off $\calW_\gamma \coloneqq \bigcup_{h \sim \gamma} S_h$,	
	\[
		\calB_\gamma(z) - (\calB_\gamma |_2 \sigma)(z) = \frac{1}{2} \sum_{h \sim \gamma} (\epsilon_h - \epsilon_{\sigma h \sigma^{-1}}) \omega_h(z) = r_\gamma(\sigma, z).
	\]
	Hence $H_\gamma \coloneqq \widetilde{F}_\gamma + \calB_\gamma$ is a locally holomorphic modular form of weight $2$ off $\calW_\gamma$.
	
	Let $\calF_Y$ be the standard fundamental domain of $\Gamma$ truncated at height $Y > \max(1, \sqrt{D_\gamma}/2)$, with positively oriented boundary. The wall set $\calW_\gamma$ meets $\partial \calF_Y$ in only finitely many points, at which $H_\gamma(z)$ has finite one-sided limits. For $m \ge 0$, we integrate $j_m(z) H_\gamma(z) \, \dd z$ piecewise along the boundary. Its invariance under the side-pairing transformations implies that the integrals over paired sides cancel, while $\calB_\gamma(z) = 0$ on the upper edge. Since the constant term of $j_m(z) \widetilde{F}_\gamma(z)$ is $\widetilde{a}_\gamma(m)$ and $j_m(z) \widetilde{F}_\gamma(z)$ is holomorphic, we obtain
	\[
		-\widetilde{a}_\gamma(m) = \int_{\partial \calF_Y} j_m(z) H_\gamma(z) \, \dd z = \int_{\partial \calF_Y} j_m(z) \calB_\gamma(z) \, \dd z.
	\]
	Since $\calF_Y$ is compact, \cref{lem:local-finite} shows that only finitely many half-disks $B_h$ meet $\calF_Y$. We therefore obtain
	\begin{align}\label{eq:boundary-Bh}
		\int_{\partial \calF_Y} j_m(z) \calB_\gamma(z) \,\dd z = \sum_{h \sim \gamma} \epsilon_h \int_{\partial \calF_Y \cap B_h} j_m(z) \omega_h(z) \, \dd z.
	\end{align}
	For each $h$, the function $j_m(z) \omega_h(z)$ is holomorphic on $\bbH$. Applying Cauchy's theorem to $\calF_Y \cap B_h$ gives
	\[
		\int_{\partial \calF_Y \cap B_h} j_m(z) \omega_h(z) \,\dd z = \epsilon_h \int_{S_h \cap \calF_Y} j_m(z) \omega_h(z) \,\dd z.
	\]
	Unfolding the conjugacy class $h = g^{-1} \gamma g$, we obtain
	\[
		\int_{\partial \calF_Y} j_m(z) \calB_\gamma(z) \, \dd z = \sum_{h \sim \gamma} \int_{S_h \cap \calF_Y} j_m(z) \omega_h(z) \,\dd z = \int_{\overline{S}_\gamma} j_m(z) \omega_\gamma (z) \,\dd z = -a_\gamma(m).
	\]
	Thus, $\widetilde{a}_\gamma(m) = a_\gamma(m)$ for every $m \ge 0$. Since the Fourier expansion of $\widetilde{F}_\gamma$ converges on $\bbH$, this coefficient identity shows that the defining series for $F_\gamma$ converges and that $\widetilde{F}_\gamma = F_\gamma$. The difference of two such solutions is a holomorphic modular form of
weight $2$ on $\Gamma$. Since no nonzero such forms exist, uniqueness follows.
\end{proof}

\begin{remark}
	Alternatively, we may adapt the contour-deformation argument in~\cite[proof of Theorem~3]{DIT2010}, which gives a direct proof without Knopp's existence theorem.
\end{remark}

\begin{proof}[Proof of \cref{thm:A}]
	Define
	\begin{align}\label{eq:F-lk-correction}
		F_\gamma^{\mathrm{lk}}(z) \coloneqq F_\gamma(z) +\frac{\pi i}{3}\Psi(\gamma)E_2(z).
	\end{align}
	By \cref{thm:modular-int-Fgamma} and~\eqref{eq:Eisen-trans}, this function is holomorphic on $\bbH$ and at $i\infty$ and satisfies~\eqref{eq:linking-trans}. \cref{prop:a-gamma} gives the constant term
	\[
		a_\gamma(0)+\frac{\pi i}{3}\Psi(\gamma)=\ell_\gamma
	\]
	and, for $m\ge1$, the Fourier coefficient 
	\[
		b_\gamma(m) =a_\gamma(m)-8\pi i\Psi(\gamma)\sigma_1(m).
	\]
	Now \cref{prop:a-gamma} and \eqref{eq:Rade-cycle} give~\eqref{eq:b-gamma}. Uniqueness follows as in the proof of \cref{thm:modular-int-Fgamma}.
\end{proof}

\section{Linking numbers of modular knots}\label{sec3}

We next evaluate the homogenized cycle integrals of $F_\gamma$ and compare them with Simon's linking-number formula. For primitive hyperbolic elements $\gamma,\sigma\in\Gamma$ of positive trace, we say that $\gamma$ and $\sigma$ are \emph{coprime} if they are not conjugate in $\Gamma$.

\subsection{Primitives and homogenized cycle integrals}\label{sec3-1}

For $\sigma\in\Gamma$ and $z\in\bbH$, define
\[
	R_\gamma(\sigma,z) \coloneqq \int_z^{\sigma z} F_\gamma(w)\,\dd w.
\]
The integral is independent of the path since $F_\gamma$ is holomorphic on the simply connected domain $\bbH$. Splitting the path at $\sigma_2z$ gives
\[
	R_\gamma(\sigma_1\sigma_2,z) =R_\gamma(\sigma_1,\sigma_2z)+R_\gamma(\sigma_2,z),
\]
and differentiation with respect to $z$ gives
\[
	\partial_zR_\gamma(\sigma,z) =(F_\gamma|_2\sigma)(z)-F_\gamma(z) =r_\gamma(\sigma,z).
\]

\begin{proposition}\label{prop:homog-cyc}
	Let $\gamma,\sigma\in\Gamma$ be coprime. Then $R_\gamma(\sigma,z)$ extends holomorphically across $w_\sigma^+$, and for every $z_0\in\bbH$,
	\[
		\lim_{N\to\infty}\int_{\sigma^Nz_0}^{\sigma^{N+1}z_0}F_\gamma(z)\,\dd z =R_\gamma(\sigma,w_\sigma^+).
	\]
	This value is independent of $z_0$ and depends only on the conjugacy classes of $\gamma$ and $\sigma$.
\end{proposition}

\begin{proof}
	If $w_\sigma^+=w_h^+$ for some $h\sim\gamma$, then $h$ and $\sigma$ also fix the algebraic conjugate of $w_\sigma^+$. The stabilizer description in \cref{sec2-1}, together with primitivity and the attracting orientation, then gives $h=\sigma$, a contradiction. Hence, for every $\delta\in\Gamma$, 
	\[
		\partial_z R_\gamma(\delta,z)=r_\gamma(\delta,z)
	\]
	is holomorphic near $w_\sigma^+$, so a local primitive extends $R_\gamma(\delta,z)$ holomorphically across $w_\sigma^+$. Since $\sigma^Nz_0\to w_\sigma^+$, the asserted limit follows.

	Conjugating $\gamma$ leaves $F_\gamma$ unchanged. For $u\in\Gamma$, the cocycle identity gives
	\[
		R_\gamma(u\sigma u^{-1},uz) = R_\gamma(\sigma,z)+R_\gamma(u,\sigma z)-R_\gamma(u,z).
	\]
	Letting $z\to w_\sigma^+$ from $\bbH$, the last two terms cancel. As $uw_\sigma^+=w_{u\sigma u^{-1}}^+$, this proves conjugacy invariance in $\sigma$ as well.
\end{proof}

The limiting value is also a homogenization in the usual averaging sense. Indeed, by splitting the path into $N$ pieces and applying Ces\`aro convergence,
\begin{align}\label{eq:Cesaro-homogenization}
	\lim_{N\to\infty}\frac1N\int_{z_0}^{\sigma^Nz_0}F_\gamma(z)\,\dd z &=\lim_{N\to\infty}\frac1N\sum_{n=0}^{N-1}R_\gamma(\sigma,\sigma^nz_0) = R_\gamma(\sigma,w_\sigma^+).
\end{align}
Thus the one-step limit and the averaged long-path integral agree here.

For comparison, Duke--Imamo\={g}lu--T\'{o}th's formula for disjoint symmetrized cycles is~\cite{DIT2017,Matsusaka2024}
\begin{align}\label{eq:sym-linking}
\frac14\mathrm{Lk}(C_\gamma+C_{\gamma^{-1}},C_\sigma+C_{\sigma^{-1}})
=\frac{1}{2\pi}\Im\lim_{N\to\infty}
\int_{\sigma^Nz_0}^{\sigma^{N+1}z_0}F_\gamma^{\mathrm{sym}}(z)\,\dd z.
\end{align}
Here the limit exists independently of $z_0$; see~\cite[Theorem~1.2]{Matsusaka2024}. In higher weights, L\"ageler--Schwagenscheidt~\cite[Theorem~1.1]{LaegelerSchwagenscheidt2022} express the imaginary parts of homogenized cycle integrals of Parson's modular integrals of weight $2k\ge4$ in terms of intersection angles of closed geodesics on $\Gamma \backslash \bbH$. This gives a higher-weight counterpart to the limiting procedure used here, while our weight $2$ formula concerns linking numbers in the $3$-manifold.

\subsection{$\matL \matR$-words}\label{sec3-2}

To evaluate these limits, we use the description of hyperbolic conjugacy classes by $\matL\matR$-words. Set
\[
	\matL = \pmat{1 & 0 \\ 1 & 1}, \qquad \matR = \pmat{1 & 1 \\ 0 & 1} = T.
\]
Every primitive hyperbolic conjugacy class of positive trace has a representative given by a word in $\matL$ and $\matR$ containing both letters, unique up to cyclic shift; see~\cite[Lemma~2.18 and Corollary~2.19]{Simon2022Thesis}. Since the class is primitive, this word is not a proper power.

Denote the set of these cyclic representatives by $\Cyc(\gamma)$, and partition it as $\Cyc(\gamma) = \Cyc_{\matL}(\gamma) \sqcup \Cyc_{\matR}(\gamma)$ according to the last letter. Set
\[
	X_\gamma \coloneqq \{w_A^+ : A \in \Cyc(\gamma)\}.
\]
Define $X_{\gamma, \matL}$ and $X_{\gamma, \matR}$ similarly from $\Cyc_{\matL}(\gamma)$ and $\Cyc_{\matR}(\gamma)$, respectively. For $h \sim \gamma$, the condition $w_h^- < 0 < w_h^+$ is equivalent to all entries of $h$ being positive. Such matrices are precisely the $\matL \matR$-words containing both letters~\cite[Lemmas~2.18 and~2.20]{Simon2022Thesis}, so the word classification above gives
\begin{align}\label{eq:Xgam-simple}
	X_\gamma = \{w_h^+ : h \sim \gamma,\ w_h^- < 0 < w_h^+\}.
\end{align}
By~\cite[Lemma~7]{Rademacher1972}, we have
\begin{align}\label{eq:Rade-word}
	\Psi(\gamma) = \# X_{\gamma, \matR} - \# X_{\gamma, \matL}.
\end{align}
Moving the last letter to the front induces the permutation
\begin{align}\label{eq:perm-phi}
	\varphi_\gamma(x) \coloneqq \begin{cases}
		x/(1+x) &\text{if } x \in X_{\gamma, \matL},\\
		x+1 &\text{if } x \in X_{\gamma, \matR}.
	\end{cases}
\end{align}
Whenever it exists, we write $f(t+i0)$ for the boundary value of a function $f$ on $\bbH$ at $t\in\R$.

\begin{lemma}\label{lem:Im-R-limits}
	For $t>0$ with $t\notin X_\gamma$, we have
	\begin{align}
		\frac{1}{2\pi}\Im R_\gamma(\matL,t+i0) &=\frac{\Psi(\gamma)}6-\#\{x\in X_{\gamma, \matR}:x<t\},\\
		\frac{1}{2\pi}\Im R_\gamma(\matR,t+i0) &=-\frac{\Psi(\gamma)}6.
	\end{align}
\end{lemma}

\begin{proof}
	By \eqref{eq:rho-gamma} and \eqref{eq:Xgam-simple},
	\[
		r_\gamma(S,z) =2\sum_{x\in X_\gamma} \left(\frac1{z-x}-\frac1{z+1/x}\right).
	\]
	Integrating with the logarithms chosen by $0<\arg<\pi$ and using $R_\gamma(S,i)=0$, since $Si=i$, we obtain
	\[
		R_\gamma(S,z) =2\sum_{x\in X_\gamma} \bigg(\log(z-x)-\log(1+xz)\bigg) -\pi i\,\#X_\gamma,
	\]
	where we used $\log(i-x)-\log(1+ix)=\pi i/2$. Moreover, the Fourier expansion of $F_\gamma$ gives
	\[
		R_\gamma(\matR,z)=R_\gamma(T,z) =\int_z^{z+1}F_\gamma(w)\,\dd w =a_\gamma(0).
	\]
	The second formula of the lemma now follows from \cref{prop:a-gamma}.

	Put $n_\gamma(u)\coloneqq\#\{x\in X_\gamma:x<u\}$. Since $\varphi_\gamma$ is a permutation of $X_\gamma$, \eqref{eq:perm-phi} gives $\varphi_\gamma(X_{\gamma, \matL})=X_\gamma\cap(0,1)$, and hence $n_\gamma (t/(1+t)) =\#\{x\in X_{\gamma, \matL}:x<t\}$. Also $t/(1+t)\notin X_\gamma$, since otherwise the injectivity of $x\mapsto x/(1+x)$ would imply $t\in X_\gamma$.

	Since $ST^{-1}S= -\matL$, the cocycle relation gives
	\[
		R_\gamma(\matL, z) =R_\gamma(S,-1-1/z)+R_\gamma(S,z)-a_\gamma(0).
	\]
	Taking imaginary parts of the boundary values in the logarithmic formula, we obtain
	\begin{align*}
		\frac{1}{2\pi}\Im R_\gamma(\matL,t+i0) =n_\gamma\left(\frac{t}{1+t}\right)-n_\gamma(t) -\frac{\Im a_\gamma(0)}{2\pi} =\frac{\Psi(\gamma)}6 -\#\{x\in X_{\gamma, \matR}:x<t\},
	\end{align*}
	where we again used \cref{prop:a-gamma}.
\end{proof}

For coprime $\gamma$ and $\sigma$, define
\begin{align*}
	C_{\matR \matL}(\gamma, \sigma) &\coloneqq \#\{(x,y) \in X_{\gamma, \matR} \times X_{\sigma, \matL}: x < y\},\\
	C_{\matL \matR}(\gamma, \sigma) &\coloneqq \#\{(x,y) \in X_{\gamma, \matL} \times X_{\sigma, \matR} : x > y\}.
\end{align*}
The combinatorics below are illustrated concretely in \cref{ex:linking-example}. 

\begin{lemma}\label{lem:CRL=CLR}
	We have $C_{\matR \matL}(\gamma, \sigma) = C_{\matL \matR}(\gamma, \sigma)$.
\end{lemma}

\begin{proof}
	Since $\varphi_\gamma$ and $\varphi_\sigma$ are permutations, the number of pairs $(x,y)\in X_\gamma\times X_\sigma$ satisfying $x<y$ is unchanged under $\varphi_\gamma\times\varphi_\sigma$. If the cyclic representatives corresponding to $x$ and $y$ have the same last letter, the inequality is preserved. If $x\in X_{\gamma, \matR}$ and $y\in X_{\sigma, \matL}$, then
	\[
		\varphi_\gamma(x)=x+1>1>\frac{y}{1+y}=\varphi_\sigma(y),
	\]
	so precisely the pairs with $x<y$, counted by $C_{\matR \matL}(\gamma,\sigma)$, are lost. Conversely, if $x\in X_{\gamma, \matL}$ and $y\in X_{\sigma, \matR}$, then
	\[
		\varphi_\gamma(x)=\frac{x}{1+x}<1<y+1=\varphi_\sigma(y),
	\]
	so precisely the pairs with $x>y$, counted by $C_{\matL \matR}(\gamma,\sigma)$, are gained. Here equality cannot occur since $X_\gamma\cap X_\sigma=\varnothing$. Hence the numbers lost and gained are equal, and therefore $C_{\matR \matL}(\gamma,\sigma)=C_{\matL \matR}(\gamma,\sigma)$.
\end{proof}

\begin{proposition}\label{prop:hcyc-CRL}
	For coprime $\gamma$ and $\sigma$ and any $z_0 \in \bbH$, we have
	\[
		\frac{1}{2\pi} \Im \lim_{N \to \infty} \int_{\sigma^N z_0}^{\sigma^{N+1} z_0} F_\gamma(z) \,\dd z = -C_{\matR \matL}(\gamma, \sigma) - \frac{1}{6} \Psi(\gamma) \Psi(\sigma).
	\]
\end{proposition}

\begin{proof}
	By the conjugacy invariance in \cref{prop:homog-cyc}, we may take $\sigma=t_1\cdots t_n$ with $t_i\in\{\matL, \matR\}$. Put $v_i=t_{i+1}\cdots t_n$. Iterating the cocycle relation gives
	\[
		R_\gamma(\sigma,z)=\sum_{i=1}^n R_\gamma(t_i,v_iz).
	\]
	Letting $z\to w_\sigma^+$ from $\bbH$ and using \cref{prop:homog-cyc}, we obtain
	\[
		\lim_{N\to\infty}
		\int_{\sigma^Nz_0}^{\sigma^{N+1}z_0}F_\gamma(z)\,\dd z
		=\sum_{i=1}^nR_\gamma(t_i,v_iw_\sigma^++i0).
	\]
	The point $v_iw_\sigma^+$ is the attracting fixed point of the cyclic representative $v_it_1\cdots t_i$, whose last letter is $t_i$. As $i$ ranges over the indices with $t_i =\matL$, these points run through $X_{\sigma, \matL}$, and similarly for $\matR$.

	Since $X_\gamma\cap X_\sigma=\varnothing$, \cref{lem:Im-R-limits} applies. The nonconstant contributions from the $\matL$-terms are
	\[
		-\sum_{y\in X_{\sigma, \matL}} \#\{x\in X_{\gamma, \matR}:x<y\} =-C_{\matR \matL}(\gamma,\sigma).
	\]
	The constant contributions are
	\[
		\frac{\Psi(\gamma)}{6} \# X_{\sigma, \matL} -\frac{\Psi(\gamma)}{6} \#X_{\sigma, \matR} = -\frac{1}{6} \Psi(\gamma)\Psi(\sigma),
	\]
	where we used \eqref{eq:Rade-word}. Combining the two contributions proves the proposition.
\end{proof}

\subsection{Simon's linking number formula}\label{sec3-3}

The linking-number formula used here was established in Simon's thesis~\cite[Proposition~4.27]{Simon2022Thesis}; see also~\cite[Proposition~5.2]{Simon2025} and the equivalent formulation in~\cite[Corollary~4.42]{Simon2022Thesis}. It computes the linking number of coprime modular knots by counting the positive crossings in the standard Lorenz diagram. In this computation, the contributing pairs of cyclic representatives are precisely those for which the order of the last letters is opposite to the order of the words themselves. Here $\matL \matR$-words are compared lexicographically, with $\matL < \matR$, after infinite repetition. This order agrees with that of their attracting fixed points. Hence the contributing pairs are precisely
\[
	A\in\Cyc_{\matR}(\gamma),\qquad B\in\Cyc_{\matL}(\sigma),\qquad w_A^+<w_B^+,
\]
and
\[
	A\in\Cyc_{\matL}(\gamma),\qquad B\in\Cyc_{\matR}(\sigma),\qquad w_A^+>w_B^+.
\]
Since all crossings in Simon's Lorenz diagram are positive, his formula therefore gives
\[
	\mathrm{Lk}_{\mathrm{Si}}(C_\gamma,C_\sigma) = \frac{1}{2} \bigg(C_{\matR \matL}(\gamma,\sigma)+C_{\matL \matR}(\gamma,\sigma)\bigg) = C_{\matR \matL}(\gamma,\sigma),
\]
by \cref{lem:CRL=CLR} (see also Rickards~\cite{Rickards2023}). Simon represents modular knots as positive Lorenz braids~\cite[Section~5.2]{Simon2025}, while we use the sign convention of Duke--Imamo\={g}lu--T\'{o}th~\cite[Section~6]{DIT2017}. The resulting linking numbers differ by a sign; for the topological background and chirality conventions, see~\cite[Section~4.1]{Simon2022Thesis}. Therefore, in our convention, 
\begin{align}\label{eq:Simon-linking}
	\mathrm{Lk}(C_\gamma,C_\sigma)=-C_{\matR \matL}(\gamma,\sigma). 
\end{align}

\begin{proof}[Proof of \cref{thm:B}] 
	Combining Simon's formula with \cref{prop:hcyc-CRL} gives 
	\[ 
		\mathrm{Lk}(C_\gamma,C_\sigma) = \frac{1}{2\pi}\Im \lim_{N\to\infty} \int_{\sigma^N z_0}^{\sigma^{N+1}z_0} F_\gamma(z)\,\dd z + \frac16\Psi(\gamma)\Psi(\sigma). 
	\] 
	By \eqref{eq:F-lk-correction} and \eqref{eq:classical-hcyc}, the right-hand side is exactly
	\[
	\frac{1}{2\pi}\Im \lim_{N\to\infty} \int_{\sigma^N z_0}^{\sigma^{N+1}z_0} F_\gamma^{\mathrm{lk}}(z)\,\dd z.
	\]
	This proves the theorem.
\end{proof}

\begin{example}\label{ex:linking-example}
	Let
	\[
		\gamma = \matL^2 \matR = \pmat{1&1\\2&3}, \qquad \sigma = \matL^3 \matR = \pmat{1&1\\3&4}.
	\]
	The class of $\gamma^{-1}$ is represented by $S\gamma^{-1}S^{-1} = \matL \matR^2$. By \eqref{eq:Rade-word},
	\[
		\Psi(\gamma) = -1, \qquad \Psi(\gamma^{-1}) = 1, \qquad \Psi(\sigma) = -2,
	\]
	which are precisely the linking numbers of the corresponding modular knots with the trefoil $\calT$ by Ghys' formula~\eqref{eq:Ghys-linking}. The cyclic representatives and their attracting fixed points are as follows:
	\[
		\renewcommand{\arraystretch}{1.5}
		\begin{array}{cc||cc||cc}
			\Cyc(\gamma) & w^+
			& \Cyc(\gamma^{-1}) & w^+
			& \Cyc(\sigma) & w^+\\
			\hline
			\matL \matL \matR & (\sqrt3-1)/2
			& \matL \matR \matR & \sqrt3-1
			& \matL \matL \matL \matR & (\sqrt{21}-3)/6\\
			\matL \matR \matL & 1/\sqrt3
			& \matR \matL \matR & \sqrt3
			& \matL \matL \matR \matL & (\sqrt{21}-1)/10\\
			\matR \matL \matL & (1+\sqrt3)/2
			& \matR \matR \matL & 1+\sqrt3
			& \matL \matR \matL \matL & (\sqrt{21}+1)/10\\
			& & & &
			\matR \matL \matL \matL & (\sqrt{21}+3)/6
		\end{array}
	\]
	The pairs contributing to $C_{\matR\matL}(\gamma,\sigma)$ are $(\matL\matL\matR,\matL\matR\matL\matL)$, $(\matL\matL\matR,\matR\matL\matL\matL)$. Those contributing to $C_{\matL\matR}(\gamma,\sigma)$ are $(\matL\matR\matL,\matL\matL\matL\matR)$, $(\matR\matL\matL,\matL\matL\matL\matR)$. Thus 
	\[
		C_{\matR \matL}(\gamma,\sigma) = C_{\matL \matR}(\gamma,\sigma)=2,
	\]
	illustrating \cref{lem:CRL=CLR}. For $\gamma^{-1}$, only the pair $(\matR\matR\matL,\matL\matL\matL\matR)$ contributes to $C_{\matL\matR}(\gamma^{-1},\sigma)$, so this count equals $1$. Using \eqref{eq:Simon-linking} and \cref{lem:CRL=CLR}, we obtain
	\[
		\mathrm{Lk}(C_\gamma,C_\sigma)=-2, \qquad \mathrm{Lk}(C_{\gamma^{-1}},C_\sigma)=-1.
	\]
	Thus, with $C_\sigma$ fixed, replacing $C_\gamma$ by $C_{\gamma^{-1}}$ is not equivalent to reversing the orientation of $C_\gamma$.
\end{example}

\subsection*{Acknowledgements}

The author thanks Christopher-Lloyd Simon for helpful comments. The author also thanks SUB G\"ottingen for providing the scan of Hecke's letter and granting permission to reproduce the excerpt. This work was supported by JSPS KAKENHI (JP24K16901) and by access to OpenAI models provided through the ChatGPT for Academic Researchers program.

\subsection*{Use of AI}

ChatGPT (GPT-6 Pro) and OpenAI Codex were used for mathematical discussions and literature searches. The final manuscript and all mathematical arguments were written and independently verified by the author, who takes full responsibility for the content.

\bibliographystyle{amsalpha}
\bibliography{References} 

\end{document}